\documentclass[psamsfonts]{amsart}

\usepackage{layout}
\usepackage[makeroom]{cancel}
\usepackage{bbm}

\usepackage{amsfonts,dsfont}
\usepackage{amsmath}
\usepackage{array}
\usepackage{amssymb,bm}
\usepackage{amsthm}
\usepackage{graphicx} 
\usepackage{mathtools}
\usepackage{multicol}
\usepackage{mathrsfs}
\usepackage{enumerate,enumitem}
\usepackage[rgb,dvipsnames]{xcolor}
\usepackage{float}
\usepackage[normalem]{ulem}
\usepackage{circuitikz}
\usepackage{cite}
\usepackage[colorlinks,linkcolor=blue,citecolor=blue,pagebackref,hypertexnames=false, breaklinks]{hyperref}

\usepackage{float}

\usepackage{fancyhdr}
\allowdisplaybreaks[4]
\newtheorem{theorem}{Theorem}[section]

\newtheorem{prop}[theorem]{Proposition}
\newtheorem{thm}[theorem]{Theorem}

\newtheorem{lemma}[theorem]{Lemma}
\newtheorem{remark}[theorem]{Remark}
\newtheorem{example}[theorem]{Example}
\newtheorem{cor}[theorem]{Corollary}
\newtheorem{examples}[theorem]{Examples}

\newtheorem{lem}[theorem]{Lemma}

\newtheorem{open question}[theorem]{Open Question}
\newtheorem{c/p}[theorem]{Conjecture/Proposition}

\newcommand{\ang}[1]{\left<#1\right>} 

\newcommand{\brak}[1]{\left(#1\right)} 
\newcommand{\norm}[1]{{\left\lVert{#1}\right\rVert}}

\newcommand{\N}[1]{||#1||} 

\def\vint{\mathop{\mathchoice%
 {\setbox0\hbox{$\displaystyle\intop$}\kern 0.22\wd0%
 \vcenter{\hrule width 0.6\wd0}\kern -0.82\wd0}%
 {\setbox0\hbox{$\textstyle\intop$}\kern 0.2\wd0%
 \vcenter{\hrule width 0.6\wd0}\kern -0.8\wd0}%
 {\setbox0\hbox{$\scriptstyle\intop$}\kern 0.2\wd0%
 \vcenter{\hrule width 0.6\wd0}\kern -0.8\wd0}%
 {\setbox0\hbox{$\scriptscriptstyle\intop$}\kern 0.2\wd0%
 \vcenter{\hrule width 0.6\wd0}\kern -0.8\wd0}}%
 \mathopen{}\int}

\newcommand{\bP}{\mathbb{P}}
\newcommand{\R}{\mathbb R}
\newcommand{\C}{\mathbb C}
\newcommand{\Rn}{\mathbb R^n}

\newcommand{\X}{\mathfrak X}

\newcommand{\F}{\mathcal F}

\mathtoolsset{showonlyrefs=true}

\title{Multiplier theorems via martingale transforms}

\author{Rodrigo Ba\~nuelos}
\address{Department of Mathematics, Purdue University, West Lafayette, IN 47907}
\email{banuelos@math.purdue.edu}
\author{Fabrice Baudoin} 
\address{Department of Mathematics, University of Connecticut, Storrs, CT 06269}
\email{fabrice.baudoin@uconn.edu}
\author{Li Chen}
\address{Department of Mathematics, Louisiana State University, Baton Rouge, LA 70803}
\email{lichen@math.lsu.edu}
\author{Yannick Sire}
\address{Department of Mathematics, Johns Hopkins University, Baltimore, MD 21218} 
\email{sire@math.jhu.edu} 
\thanks{R. Ba\~nuelos supported in part by NSF Grant DMS-1854709. F. Baudoin supported in part by NSF Grant DMS-1901315}

\date{\today}

\begin{document}

\maketitle

\begin{abstract}
We develop a new and general approach to prove multiplier theorems in various geometric settings. The main idea is to use martingale transforms and a Gundy-Varopoulos representation for multipliers defined via a suitable  extension procedure. Along the way, we provide a probabilistic proof of a generalization of a result by Stinga and Torrea, which is of independent interest. Our methods here also recover the sharp $L^p$ bounds for second order Riesz transforms by a liming argument. 
\end{abstract}

\tableofcontents

\section{Introduction and main results} The $L^p$-boundedness properties of Riesz transforms in wide geometric settings have been extensively studied by a large number of authors for many  years.  The large literature on this topic includes techniques from the Calder\'on-Zygmund theory of singular integrals and  probabilistic and analytic Littlewood-Paley theory.  For some of this literature we refer the reader to \cite{ACDH04},  \cite{Bak87} and \cite{BBC}.  On the other hand, the probabilistic approach of  R.~F.~Gundy and N.~Th.~Varopoulos \cite{GV79} which represents the Riesz transforms as conditional expectations of martingale transforms, combined  with the sharp martingale inequalities of D.L.~Burkholder,  provides a powerful tool to obtain not only $L^p$ bounds with constant that do not depend on the geometry of the ambient space but often give sharp, or nearly sharp, bounds.  The martingale techniques also apply to Riesz transforms on Wiener space providing explicit bounds.  For  an incomplete list of references  to this now very large literature, we refer to \cite{BBC} and \cite{BO15}. In addition to providing universal and explicit $L^p$ bounds,  the martingale transform techniques extend to multipliers beyond Riesz transforms. For some of this literature, we refer to \cite{Ban}.  A common thread in the Gundy-Varopoulos constructions has been to build the martingale transforms on stochastic processes of the form $(X_t, Y_t)$ where $X_t$ is either a diffusion or a process arising  from a Markovian semigroup on  $\Rn$ or on a manifold $M$ (such as the L\'evy multipliers studied in \cite{BanBog}),  and where $Y_t$ is either a one dimensional Brownian motion on $\R^{+}$ killed upon hitting $0$ (harmonic extensions) or $T-t$ for some fixed time $T$, in the case of space-time (heat extension) constructions  as in \cite{BaBau13}.   The goal of this paper is to prove  boundedness of multipliers obtained when the ``vertical" process $Y_t$ is more general than those just mentioned. More precisely, we will study multipliers that arise as conditional expectations of martingale transforms which are built on the process $(X_t, \eta_t)$ where the vertical diffusion has a generator of the form given in \eqref{eq:1dim}.  As we show in Section  4 (see Remark \ref{limsigma}), our construction unifies both the original constructions with  $(X_t, Y_t)$ of Gundy-Varopoulos, which gives  sharp inequalities for first order Riesz transforms \cite{BW}, and the construction for $(X_t, T-t)$ from  \cite{BM},  which gives sharp inequalities for second order Riesz transforms,  into one by a limiting procedure.

The last two decades or so have seen a great amount of works dealing with nonlocal operators (generators of L\'evy processes) from the PDE point of view (see e.g. the recent book \cite{bookNL}). In particular, the paper \cite{caffaSil} has been instrumental in interpreting fractional powers of the Laplacian in $\mathbb R^n$ in terms of a suitable ``harmonic" extension. Note that in the language of probability, this result had been proved in \cite{MolOst1969}. This latter result has been put in a more general (and flexible) framework by Stinga and Torrea in \cite{ST10}. It is beyond the scope of this paper to review the amount of works using such technique.  Our contributions here lie at the interface of probabilistic methods and harmonic analysis. More precisely, in the present paper, combining the Gundy-Varopoulos approach to Riesz transforms and a probabilistic approach
to the result of Stinga and Torrea, we obtain new results  about the boundedness in $L^p$ of  three types of operators: 

\begin{itemize}
\item Multipliers of the type $\Phi (-\Delta+V)$, where $\Delta$ is a diffusion operator and $V$ a non-negative bounded smooth potential;
\item Generalized Riesz  transforms of the type $\Phi (-\Delta+V) \X_i $, where the $\X_i $'s are first-order differential operators that commute with $-\Delta+V$;
\item Generalized second order Riesz  transforms of the type $\Phi (-\Delta+V) \X_i \X_j $.
\end{itemize}

 We note that using methods from harmonic analysis  many results about $L^p$-estimates for  Schr\"odinger operators are already available in the literature in some settings and various more general assumptions on the potential $V$, see for instance \cite{MR2965364,Shen}. However, those methods rely heavily on the geometry of the underlying space and yield dimension dependent $L^p$-bounds.  The probabilistic method we are using here uses stronger assumptions on the potential $V$ but on the other hand yields dimension independent $L^p$-bounds and relies very little on the geometry of the underlying space.   

When $V=0$ among other things, we prove the following general multiplier theorem. Let $\Delta$ be a locally subelliptic   (in the sense of Fefferman-Phong) diffusion operator on a  smooth manifold $M$ which is essentially self-adjoint on the space of smooth and compactly supported functions with respect to a measure $\mu$ on $M$. We assume that $\Delta$ generates a diffusion process $((X_t)_{t \ge 0}, (\mathbb{P}_x)_{x \in M})$ which is not explosive. If $\Phi$ is a bounded Borel function on $[0,+\infty)$ the operator $\Phi(-\Delta)$ may be defined on $L^2(M,\mu)$ by using the spectral theorem. By using martingale transforms, we will then prove the following theorem.

\begin{theorem}\label{gen multiplier}
 If there exists a finite complex Borel measure $\alpha$ on $\mathbb{R}_{\ge 0}$ such that for every $x \in [0,+\infty)$,
\begin{align}\label{stieltjes}
\Phi(x)= \int_{0} ^{+\infty} \left( 1-\frac{m}{\sqrt{m^2+x}}\right) d \alpha (m),
\end{align}
then, for every  $p>1$ and $f \in L^p (M,\mu)$,
\[
\left\| \Phi(-\Delta) f\right\|_p \le 2(p^*-1) | \alpha |([0,+\infty))  \|f\|_p,
\]
where $p^*=\max\{p,\frac{p}{p-1}\}$.

\end{theorem}

In Theorem \ref{bounded 2} below, we actually prove a more general result that also applies to Schr\"odinger operators.
The representation \eqref{stieltjes} is related to the theory of Stieltjes transforms, see \cite{MR0073746,MR1501933}, and is possible to invert.  We note that Theorem \ref{gen multiplier} can also be proved  using Bernstein theorem, since the function $x \to \frac{m}{\sqrt{m^2+x}}$ is completely monotone. However, the method we propose is general and is easily adapted to study different multipliers as generalized first order or second order Riesz transforms.

Concerning the study of generalized first order and second order Riesz transforms on Lie groups of compact type, using a variation of the method to construct multipliers, we obtain the following result.
 
 \begin{theorem}\label{compact lie 1}
 Let $G$ be a $n$-dimensional  Lie group of compact type endowed with a bi-invariant
Riemannian structure. Let $\X_1,\cdots,\X_n$ be an orthonormal frame of the Lie algebra of $G$ and denote by $\Delta$ the Laplace Beltrami operator on $G$. Let $\Phi: (0,+\infty) \to \mathbb{C}$ be a complex Borel function.

\begin{enumerate}
\item

If there exists a finite complex Borel measure $\alpha$ on $[0,+\infty)$ such that for every $x \in (0,+\infty)$,
\begin{align}\label{st2}
\Phi(x)= \int_{0} ^{+\infty} \frac{d \alpha (m)} {\sqrt{ x +m}},
\end{align}
then, for every $1 \le i \le n$, $p>1$, and $f \in L^p$
\begin{align}\label{bound intro 1}
\left\| \Phi(-\Delta) \X_i f\right\|_p \le \cot\left(\frac{\pi}{2p^*}\right)  |\alpha |([0,+\infty))  \|f\|_p.
\end{align}

\item If there exists a finite complex Borel measure $\alpha$ on $[0,+\infty]$ such that for every $x \in (0,+\infty)$
\begin{align*}
\Phi(x)= \int_{0} ^{+\infty} \frac{d \alpha (m)} {\sqrt{ x +m^2} ( \sqrt{ x +m^2} -m) },
\end{align*}
then, for every $1 \le i,j \le d$, $p>1$, and $f \in L^p$
\begin{align}\label{bound intro 2}
\left\| \Phi(-\Delta) \frac{1}{2}( \X_i \X_j +\X_j \X_i) f\right\|_p \le (p^*-1)  |\alpha |([0,+\infty])  \|f\|_p.
\end{align}
\end{enumerate}
 \end{theorem}
 
 Theorem \ref{compact lie 1} is sharp. Indeed, in \eqref{bound intro 1}, if one choses $\alpha$ to be the Dirac distribution at $0$, one gets
 \[
 \left\| (-\Delta)^{-1/2}  \X_i f\right\|_p \le \cot\left(\frac{\pi}{2p^*}\right)   \|f\|_p
 \]
 which is the sharp bound for the Riesz transform, see \cite{BW} and   \cite{IM}. In \eqref{bound intro 2}, if one choses $\alpha$ to be the Dirac distribution at $+\infty$, one gets
 \[
 \left\| (-\Delta)^{-1} \frac{1}{2}( \X_i \X_j +\X_j \X_i) f\right\|_p \le \frac{1}{2} (p^*-1)    \|f\|_p
 \]
 which is the sharp bound for the second order Riesz transform, see \cite{BM} and \cite{GeiSmiSak}.

\

Finally, using techniques developed in \cite{BBC} to handle the study of Riesz transforms on vector bundles, we obtain the following result.
 
 \begin{theorem}\label{Riesz Hodge}
 Let $M$ be a  complete Riemannian manifold with non-negative Weitzenb\"ock curvature. Let $\mathcal{L} =d d^* +d^* d$ be the Hodge-de Rham Laplace operator on the exterior bundle of $M$. Let $\Phi: (0,+\infty) \to \mathbb{C}$ be a complex Borel function. If there exists a finite complex Borel measure $\alpha$ on $\mathbb{R}_{\ge 0}$ such that for every $x \in (0,+\infty)$,
\[
\Phi(x)= \int_{0} ^{+\infty} \frac{d \alpha (m)} {\sqrt{ x +m}},
\]
then, for every  $p>1$ and   every $L^p$ integrable exterior differential form $\eta$
\[
\|  \Phi(\mathcal L ) \, d \eta \|_p \le 6 (p^*-1) |\alpha |([0,+\infty))  \| \eta \|_p.
\]
 \end{theorem}

\section{Preliminaries, Extension procedure}\label{sec2}

\subsection{Setting}\label{sec2.1}

Let $\Delta$ be a locally subelliptic diffusion operator  (see Section 1.2 in \cite{baudoin2018geometric} and \cite{MR922334} for a definition of local subellipticity) on a smooth manifold $M$. For every smooth functions $f,g: M \rightarrow \mathbb{R}$, we define the so-called \textit{ carr\'e du champ} operator, which is   the  symmetric first-order differential form defined by:
\[
\Gamma (f,g) =\frac{1}{2} \left( \Delta(fg)-f\Delta g-g\Delta f \right).
\]
A straightforward computation shows that  in a local chart one has
\[
\Delta=\sum_{i,j=1}^n \sigma_{ij} (x) \frac{\partial^2}{ \partial x_i \partial x_j} +\sum_{i=1}^n b_i (x)\frac{\partial}{\partial x_i},
\]
where $\left(\sigma_{ij} (x)\right)$ is symmetric and nonnegative.  That is, for $1\le i,j\le n$, $\sigma_{ij} (x)=\sigma_{ji} (x)$ and for $\xi\in \Rn$, $\sum_{i, j=1}^n\sigma_{ij} (x)\xi_i\xi_j\geq 0$. Then in the same chart
\[
\Gamma (f,g)=\sum_{i,j=1}^n  \sigma_{ij} (x) \frac{\partial f}{\partial x_i} \frac{\partial g}{\partial x_j}.
\]
As a consequence, for every smooth function $f$, $\Gamma(f):=\Gamma(f,f) \ge 0$. We assume that $\Delta $  is symmetric with respect to some smooth measure $\mu$, which means that for every smooth and compactly supported functions $f,g \in C_0^\infty(M)$,
\[
\int_{M} g \Delta f d\mu= \int_{M} f \Delta g d\mu.
\]
By smooth measure $\mu$, we mean here that $\mu$ is given by a density in the sense of Definition 3.90 in \cite{Ga-Hu}. That is,  locally, in any coordinate system $x_i$,  $\mu$ has a smooth density with respect to the volume form $|dx_1\wedge \cdots \wedge dx_n|$.
There is an intrinsic distance associated to the operator $\Delta$ which is defined by
\[
d(x,y)=\sup \left\{ |f(x) -f(y) | , f \in  C^\infty(M) , \| \Gamma(f) \|_\infty \le 1 \right\},\ \ \  \ x,y \in M.
\]
We assume that the metric space $(M,d)$ is complete. In that case, from Propositions 1.20 and 1.21 in \cite{baudoin2018geometric}, the operator $\Delta$ is essentially self-adjoint on $C_0^\infty(M)$.

  Let now  $V:M\to \R$ be a non-positive lower bounded smooth potential and consider the Schr\"odinger operator
  \[
  L=\Delta+V.
  \]

 The operator $L$ is also  essentially self-adjoint on the space of smooth and compactly supported functions. Indeed from Proposition 1.21 in \cite{baudoin2018geometric}, there exists a sequence $h_n\in C_0^\infty(M)$, $0 \le h_n \le 1$ such that $h_n \nearrow 1$ and $ \| \Gamma(h_n) \|_\infty \to 0$. Using then the argument in the proof of  Proposition 1.20 in \cite{baudoin2018geometric} together with the fact that $V \le 0$ yields the fact that $L$ is   essentially self-adjoint on $C_0^\infty(M)$. The  self-adjoint extension of $L$ will still be denoted by $L$. The semigroup in $L^2(M,\mu)$ generated by $L$ will be denoted by $(P_t)_{t \ge 0}$. 
 
 We assume that $\Delta$ generates a diffusion process $(X_t, (\mathbb{P}_x)_{x \in M})$ which is not explosive. In that case, the Schr\"odinger semigroup $(P_t)_{t \ge 0}$ admits the Feynman-Kac representation (see for instance \cite[Theorem 6.20]{Baudoin2014}):
 \begin{align}\label{FK}
 P_t f (x) =\mathbb{E}^x \left( e^{\int_0^t V(X_s) ds} f(X_t) \right), \quad f \in C_0^\infty(M).
 \end{align}
 The semigroup $(P_t)_{t \ge 0}$ hence defined is then a sub-Markov semigroup (see page 71 in \cite{Baudoin2014} for a definition and basic properties of sub-Markov semigroups).
 
 \begin{remark}
 It is a well-known result by A. Grigor'yan \cite[Theorem 1]{MR860324} and K.-T. Sturm \cite[Theorem 4]{MR1301456} that a sufficient condition for $\Delta$ to generates a diffusion process $(X_t, (\mathbb{P}_x)_{x \in M})$ which is not explosive is that for some $x_0 \in M$ and $r_0>0$
 \[
 \int_{r_0}^{+\infty} \frac{ r \, dr}{\ln \mu (B(x_0,r))} =+\infty,
 \]
 where $B(x_0,r)$ denotes the metric ball with radius $r$ for the distance $d$. This is for instance satisfied if for some constants $C_1,C_2>0$ one has  $\mu (B(x_0,r))\le C_1e^{C_2 r^2}$.
 \end{remark}
 
\subsection{Green function at $+\infty$  of one-dimensional diffusions killed at 0}

Let $a,b$ be  smooth functions on $(0,\infty)$ with $a>0$ and let
\[s'(z)=\exp \left(- \int_1^z \frac{b(y)}{a(y)^2} dy  \right).\] Assume that
\begin{equation}\label{eq:cond}
 \int_1^{\infty}s'(z)dz=\infty,
\quad
 \int_0^{1} s'(z)dz <\infty.
\end{equation}

We consider a one-dimensional diffusion operator on $(0,+\infty)$
\begin{equation}\label{eq:1dim}
\mathcal B = a(y)^2 \frac{\partial^2 }{\partial y^2} + b(y) \frac{\partial }{\partial y},
\end{equation}
with Dirichlet boundary condition at 0. The quantities
$s'(z)$ and $m(z):=\frac{1}{s'(z) a(z)^2}$ are respectively often called the scale function and density of the speed measure associated with the diffusion $\mathcal B$. For more on this, see \cite[Section II.9]{BorSal} or  \cite[Chapter VII, Definitions 3.3 and 3.7]{RevuzYor}.

Let $\eta_t$ be the diffusion process with generator $\mathcal B$.
We  denote
\[
\tau =\inf \{ t >0, \eta_t=0\}.
\]
and $q_t (y)$ the density of $\tau $ under $\mathbb P_y$, $\eta_0=y>0$. It is well known that under the assumption \eqref{eq:cond}, the process $\eta$ is not explosive and hits zero with probability 1  (see for instance \cite[Ch VII Proposition 3.2]{RevuzYor}), that is, 
\[
\mathbb P(\tau<+\infty)=1.
\]
For later use, we assume that $\eta$ can be written as a (weak) solution of a SDE
\begin{align}\label{SDEeta1}
d\eta_t=b(\eta_t) dt+  a(\eta_t) d\beta_t, \quad t < \tau,
\end{align}
where $\beta_t$ is a Brownian motion on $\R$ with $\mathbb E(\beta_t^2)=2t$, which is independent of the process $(X_t)_{t \ge 0}$. 
We first collect some preliminary results about the Green function at $+\infty$ of the diffusion $\eta$ killed at 0. 
For computations, it is convenient to write $\mathcal B$ as
\begin{equation}\label{generatoreta2}
\mathcal B =a(y)^2  \frac{\partial^2 }{\partial y^2} + a(y)^2 \frac{h'(y)}{h(y)} \frac{\partial }{\partial y},
\end{equation}
where $h$ is a nonnegative function such that $a(y)^2 \frac{h'(y)}{h(y)}=b(y)$. Note that one can choose
\[
h(y) = \exp \left( \int_1^y \frac{b(w)}{a(w)^2} dw  \right)
\]
so that the  assumptions in \eqref{eq:cond} imply 
\[
\int_0^1 \frac{dw}{h(w)} <+\infty , \quad \int_1^{+\infty} \frac{dw}{h(w)} =+\infty.
\]

The following lemma  which computes  the Green function of $\mathcal B$ on the half-line $[0,+\infty)$ with Dirichlet boundary condition at 0 is then straightforward. 

\begin{lem}\label{lem:Green}
Let $g$ be a  Borel function such that $\int_{0}^{+\infty} h(z) \frac{|g(z)|}{a(z)^2}dz <+\infty$. The solution on $[0,+\infty)$ of the equation
\[
\mathcal{B} f =-g
\]
with boundary conditions $f(0)=0$ and $(f' h)(+\infty)=0$,  is given by
\[
f(y)=\int_0^{+\infty}G(y,z) g(z) dz,
\]
where
\[
G(y,z)=\frac{h(z)}{a(z)^2} \int_0^{z \wedge y} \frac{dw}{h(w) } .
\]
In particular,
\[
G(+\infty,z):=\lim_{y \to +\infty} G(y,z)=\frac{h(z)}{a(z)^2} \int_0^{z } \frac{dw}{h(w) }=s(z)m(z).
\]
\end{lem}

\begin{proof}Notice that the equation $\mathcal B f=-g$ can be rewritten as
\[
a(y)^2 f''(y)+ a(y)^2 \frac{h'(y)}{h(y)} f'(y)=-g(y),
\]
where $h(y) = \exp \left( \int_1^y \frac{b(w)}{a(w)^2} dw  \right)$.
This is equivalent to
\[
\frac{1}{h(y)} ( f' h )' (y)=-\frac{g(y)}{a(y)^2}.
\]
Since $\int_{0}^{+\infty} h(z) \frac{|g(z)|}{a(z)^2}dz <+\infty$, the first order ODE with boundary condition $(f' h)(+\infty)=0$ has a unique solution
\[
f'(y) h(y)=\int_{y}^{+\infty} h(z) \frac{g(z)}{a(z)^2}dz.
\]
Again from the boundary condition $f(0)=0$, we conclude the unique existence of the solution  $f$ as
\begin{align*}
f(y) &=\int_0^y \frac{1}{h(z) } \int_{z}^{+\infty} h(w) \frac{g(w)}{a(w)^2} dw \, dz \\
 &=\int_0^{+\infty} \int_0^{w \wedge y} \frac{dz}{h(z) } \, \frac{h(w)}{a(w)^2} g(w) dw.
\end{align*}
\end{proof}

Our next lemma is the occupation time formula for the process $\eta_t$. 

\begin{lem}\label{lem:eqB}
Let $G$ be the Green function of $\mathcal B$ on the half-line $[0,+\infty)$ with Dirichlet boundary condition at 0 as above. Then, if  $g$ is a positive  Borel function such that $\int_{0}^{+\infty} h(z) \frac{g(z)}{a(z)^2}dz <+\infty$, for every $y>0$,
\[
\mathbb{E}_{y} \left(  \int_0^\tau g(\eta_s) ds \right)=f(y),
\]
where $f$ solves the equation in Lemma \ref{lem:Green}.
\end{lem}

\begin{proof} 
Let $f$ be the solution of
\[
\mathcal{B} f =-g
\]
with boundary conditions $f(0)=0$ and $(f' h)(+\infty)=0$. By It\^o's formula,
\[
f(\eta_t)=f(\eta_0)+\int_0^t f'(\eta_s)a (\eta_s) d\beta_s+\int_0^t \mathcal B f(\eta_s) ds, \quad t < \tau.
\]
In particular, letting $t \to \tau$, one obtains
\[
f(\eta_0)=\int_0^{\tau} g(\eta_s) ds-\int_0^\tau f'(\eta_s)a (\eta_s) d\beta_s.
\]
Denote by $\tau_n=\tau \wedge \sigma_n \wedge n$, where $\sigma_n =\inf \{ t \ge 0, \eta_t=n \}$.   Since $a$ and $f'$ are smooth in $(0, \infty)$ and $\tau_n$ is a bounded stopping time, $\left(\int_0^{t\wedge \tau_n} f'(\eta_s)a(\eta_s) d\beta_s\right)_{t\ge 0}$, is a martingale. Applying  the Doob's optional stopping theorem  we get
\[
\mathbb E^y\left(\int_0^{\tau_n} f'(\eta_s)a(\eta_s) d\beta_s\right)=\mathbb E^y\left(f(\eta_{\tau_n})-f(\eta_0)-\int_0^{\tau_n} \mathcal B f(\eta_s) ds\right)=0.
\]
This gives
\[
f(y)=\mathbb E^y\left( \int_0^{\tau_n} g (\eta_s) ds \right) + \mathbb E^y\left(f(\eta_{\tau_n}) \right).
\]
Letting $n\to \infty$, the monotone convergence theorem yields 
\[
f(y)=\mathbb E^y\left(\int_0^{\tau} g(\eta_s) ds\right).
\]
\end{proof}

\subsection{Extension procedure with general vertical diffusions}

If $f \in L^2 (M,\mu)$  we consider its extension to the cone $M \times [0,+\infty)$ defined for $ x \in M, y \in [0,+\infty)$ by
\begin{align}\label{extension_def}
U_f (x ,y)=\int_0^{+\infty} P_t f (x) q_t(y) dt, 
\end{align}
where we recall that $P_t=e^{tL}$ is the  semigroup generated by $L=\Delta+V$ and that $q_t$ is the density of the first hitting time $\tau$ of zero by $\eta$. Since $\Delta$ is locally subelliptic and $V$ is smooth, using the definition of local subellipticity for $\Delta$ (see Definitions 1.6 and 1.8 in \cite{baudoin2018geometric} ) one deduces that $L$ is itself  locally subelliptic. Therefore $P_t f$ is a smooth function for every $f \in L^2 (M,\mu)$ (see Proposition 1.23 in \cite{baudoin2018geometric}).

Denoting
\begin{equation}\label{eq:extension}
\mathcal{K} (y , \lambda):=\int_0^{+\infty} e^{-\lambda t} q_t(y) dt=\mathbb{E}^y (e^{-\lambda \tau}),
\end{equation}
we have that  $\mathcal{B}\mathcal{K} (\cdot,\lambda )=\lambda \mathcal{K} (\cdot,\lambda )$. In particular, since $\mathcal{B}$ is an elliptic operator, we deduce that $\mathcal{K} (\cdot,\lambda )$ is a smooth function (see page 18 in \cite{BorSal} for more details, including a representation formula for  $\mathcal{K} (\cdot,\lambda )$).

We note that  the spectral theorem shows that in the $L^2$ sense
\[
U_f(x,y)=\mathcal{K} (y , -L)f (x).
\]

The starting point of our approach is the following generalization of a result by Stinga and Torrea (see \cite{ST10}).

\begin{thm}\label{stinga-torrea}
Let  $f \in C_0^\infty(M)$. In the pointwise sense $U_f$ satisfies
\begin{equation}\label{extL}
\begin{cases}
 (L+ \mathcal B) U_f = 0\ \ \ \ \ \ \ \text{in} \, \, M \times (0,+\infty)
\\
U(\cdot ,0) = f \ \ \ \ \ \ \ \ \ \ \ \ \text{on} \, \, M.
\end{cases}
\end{equation}
\end{thm}

We shall give a probabilistic proof of this result which is based on a martingale that shall be used several times in this paper.

\begin{lem}\label{martin}
Let $f \in C_0^\infty(M)$. Consider  the process
\[
M_{t}^f=e^{\int_0^{t \wedge \tau} V(X_u)du}U_f (X_{t \wedge \tau},\eta_{t \wedge \tau}).
\]
The process $M_{t}^f$ is a martingale with quadratic variation 
\[
\langle M^f\rangle_t=2 \int_0^{t \wedge \tau} e^{2\int_0^{s} V(X_u)du} \Gamma (U_{f}) (X_s,\eta_s) ds+2\int_0^{t \wedge \tau} e^{2\int_0^{s} V(X_u)du} \partial_yU_f (X_s,\eta_s)^2 a(\eta_s)^2 ds,
\]
where $\Gamma (U_{f})=\sum_{i,j=1}^n  \sigma_{ij} \frac{\partial U_{f}}{\partial x_i} \frac{\partial U_{f}}{\partial x_j}$.
\end{lem}

\begin{proof}
First note that 
\[
M_{\tau}^f=e^{\int_0^{\tau} V(X_u)du} f (X_{\tau}).
\]
Assume that the process $(X_t,\eta_t)_{t\ge 0}$ starts at $(x,y)\in M\times (0,\infty)$. Since the processes $X_t$ and $\eta_t$ are independent, it follows from the Feynman-Kac formula that 
\begin{align*}
\mathbb{E}^{x,y}\left(e^{\int_0^{\tau} V(X_u)du} f(X_{\tau} ) \right)
&=
\int_0^{\infty} \mathbb{E}^x\left(e^{\int_0^{s} V(X_u)du} f(X_{s} ) \right)q_s(y) ds
\\ &=
\int_0^{\infty} P_sf(x) q_s(y) ds=U_f(x,y),
\end{align*}
where we recall that $q_s(y)$ is the density of $\tau$ under $\eta_0=y>0$. Denote by $(\mathcal F_t)_{t\ge 0}$ the natural filtration of $(X_t,\eta_t)$.
From the strong Markov property we have 
\begin{align*}
\mathbb{E}\left(M_{\tau}^f \mid \mathcal F_{s \wedge \tau}\right)
&=\mathbb{E}\left(e^{\int_0^{\tau} V(X_u)du} f(X_\tau) 1_{\tau\le s}\mid \mathcal F_{s \wedge \tau} \right)+\mathbb{E}\left(e^{\int_0^{\tau} V(X_u)du} f(X_{\tau}) 1_{\tau>s}\mid \mathcal F_{s \wedge \tau}\right)
\\ &
=e^{\int_0^{\tau} V(X_u)du}f(X_\tau) 1_{\tau\le s}+e^{\int_0^{s\wedge \tau} V(X_u)du}U_f (X_{s \wedge \tau},\eta_{s \wedge \tau})  1_{\tau>s}
\\ &
=M_{s\wedge \tau}^f .
\end{align*}
We conclude that $M_{t}^f $ is a martingale. The quadratic variation of $M_{t}^f $ is computed as in  \cite[p.324]{RevuzYor} or \cite[p.~181]{BakEme1}. Indeed, from It\^o's formula,  the bounded variation part of $M_t^f$ is zero. Hence

\begin{equation}\label{eq:Ito}
\begin{split}
M_{t}^f=U_f(x,y)+\sum_{i=1}^n \int_0^{t\wedge \tau} e^{\int_0^{s}V(X_u)du} \bigg( \sum_{j=1}^n v_{ij}\partial_{x_j} \bigg)U_f(X_s, \eta_s)dB_s^i \\+ \int_0^{t\wedge \tau} e^{\int_0^{s}V(X_u)du} \partial_y U_f(X_s, \eta_s)a(\eta_s)d\beta_s,
\end{split}
\end{equation}
where $(v_{ij}(x))$ is the square root of the symmetric nonnegative matrix $(\sigma_{ij}(x))$ and $B_t=(B_t^1,\cdots,B_t^n)$ is a Brownian motion on $\R^n$ with generator $\Delta=\sum_{i=1}^n \frac{\partial^2}{\partial x_i^2}$. Thus $\langle M^f\rangle_t$ immediately follows.
\end{proof}

We are now in position to prove Theorem \ref{stinga-torrea}. 

\begin{proof}[Proof of Theorem \ref{stinga-torrea}]
Since $M_{t}^f=e^{\int_0^{t\wedge \tau} V(X_u)du}U_f (X_{t \wedge \tau},\eta_{t \wedge \tau})$ is a martingale, it follows from It\^o's formula that the bounded variation part of $M_t^f$ is zero, i.e.,
\[
\int_0^{t\wedge \tau} e^{\int_0^{s} V(X_u)du} (L + \mathcal B) U_f (X_s, \eta_s) ds=0.
\] 
We conclude that 
\[
 (L + \mathcal B) U_f(x,y) =\lim_{t\to 0} \frac1t\int_0^{t\wedge \tau} e^{\int_0^{s} V(X_u)du} (L + \mathcal B) U_f (X_s, \eta_s) ds=0.
\]
\end{proof}

\subsection{Martingale inequalities}

In this section, we recall some results on the martingale inequalities used in this paper. 
Suppose that $(\Omega, \F, \bP)$ is a complete probability space, filtered by $\F=\{\F_t\}_{t\ge 0}$, a family of right continuous sub-$\sigma$-fields of $\F$. Assume that $\F_0$ contains all the events of probability $0$. Let $X$ and $Y$ be adapted, real-valued martingales which have right-continuous paths with left-limits (r.c.l.l.).  The martingale $Y$ is differentially subordinate to $X$ if $|Y_0|\le |X_0|$ and $\ang{X}_t-\ang{Y}_t$ is a nondecreasing and nonnegative function of $t$.  The martingales $X_t$ and $Y_t$ are said to be orthogonal if  the covariation process $\ang{X, Y}_t=0$ for all $t$.  

We always assume that the martingales are $L^p$ bounded for $1<p<\infty$ and by $X$ we mean $X_{\infty}$.  By $\|X\|_{p}$ we  mean $\sup_{t>0}\|X_t\|_p=\|X_{\infty}\|_p$.  This is often applied to stopped martingales. Hence $\|X\|_p=\|X_{\tau}\|_p$ where $\tau$ is a stopping time.  We use the notation $\|f\|_p$ for the $L^p$-norms of functions defined on $M$ with respect to the measure $\mu$. This is clear in each occurrence  and should not create any confusion. 

In the following, we recall the sharp inequalities of martingale transforms proved by Ba\~nuelos and Wang \cite{BW}, as well as an extension by Ba\~nuelos and  Os{\c e}kowski \cite{BO15}. 
\begin{thm}[\protect{\!\!\cite[Theorems 1 and 2]{BW}}]\label{thm:BW95}
Let $X$ and $Y$ be two martingales with continuous paths such that $Y$ is differentially subordinate to $X$. Fix  $1<p<\infty$ and set $p^*=\max\{p,\frac{p}{p-1}\}$.  Then 
\begin{equation*}\label{eq:MT}
\|Y\|_p \le (p^*-1)\|X\|_p.
\end{equation*}
Furthermore, suppose that the martingales $X$ and $Y$ are orthogonal.  Then 
\begin{equation*}\label{eq:OMT}
\|Y\|_p \le \cot\!\brak{\frac{\pi}{2 p^*}}\|X\|_p. 
 \end{equation*}
Both of these inequalities are sharp.
\end{thm}
\begin{thm}[\protect{\!\cite[Theorem 2.2]{BO15}}]\label{thm:BO15}
Let $X$ and $Y$ be two martingales with continuous paths such that $Y$ is differentially subordinate to $X$.
Consider the process 
\[
Z_t=e^{\int_0^{t} V_sds }\int_0^t e^{-\int_0^{s} V_vdv }dY_s,
\]
where $(V_t)_{t\ge 0}$ is a non-positive adapted and continuous process. 
For $1<p<\infty$, we  have the sharp bound
\begin{equation*}\label{MT}
\|Z\|_p \le (p^*-1)\|X\|_p.
\end{equation*}
\end{thm}

\section{Multiplier theorems}

The martingale transform method to construct multipliers is very versatile and allows one to deal with a very general setup.  We work under the assumptions and with the notations of Section 2.  In particular, let $G$ be the Green function of $\mathcal B$ on the half-line $[0,+\infty)$ with Dirichlet boundary condition at 0 (see Lemma \ref{lem:Green}).  
Before going further, in order to put the next results in perspective, we discuss  the well-known H\"ormander-Mikhlin theorem in $\mathbb R^n$. Let us assume that $\Delta=\sum_{i=1}^n \frac{\partial^2}{\partial x_i^2}$ is the Laplace operator in $\mathbb{R}^n$ and that $V=0$. Recall the H\"ormander-Mikhlin theorem:
\begin{thm}[\protect{\!\!\cite[Theorem 5.2.7]{grafakos}}]\label{T:HormanderMult}
Let $m\in L^\infty(\mathbb R^n \backslash \{ 0 \})$ be a complex-valued bounded function on $\mathbb R^n \backslash \{ 0 \}$ that satisfies:
\begin{itemize}
\item[(a)] Either the Mikhlin's condition 
$$
|\partial ^\gamma_\xi m(\xi) | \leq A |\xi|^{-|\gamma|}
$$
for all multi-indices $|\gamma| \leq \Big [ \frac{n}{2} \Big ]+1.$
\item[(b)] Or the H\"ormander's condition
$$\sup_{R>0} R^{-n+2|\gamma|} \int_{R < |\xi| <2R} |\partial ^\gamma_\xi m(\xi) |^2 \,d\xi \leq A^2 <\infty 
$$
for all multi-indices $|\gamma| \leq \Big [ \frac{n}{2} \Big ]+1.$
\end{itemize}
Then the operator $T_m$ whose symbol is $m$ is bounded in $L^p (\mathbb R^n)$ for any $p \in (1,+\infty)$, with an operator norm depending on $n,A,p,\|m\|_\infty$.
 \end{thm}
 
 \begin{remark}
 We point out  there are sharper versions of the previous theorem replacing the Lebesgue spaces by Lorentz spaces (see e.g. \cite{grafakossharp}) and also more general variants (see e.g. \cite{DSY}).
 \end{remark}
 
 In our framework, we want to study boundedness properties of the multiplier $\Phi (-L)$ where
\begin{equation*}
\Phi(\lambda) = \int_0^{+\infty} G(+\infty,y) \partial_y\mathcal{K} (y , \lambda)^2  a(y)^2 dy. 
\end{equation*}

From now on and throughout the paper we will make  the following standing assumptions on the functions involved in $\Phi$:
\begin{itemize}
\item $a$ is bounded on $(0,+\infty)$;
\item The functions $G(+\infty,y)$ and $G(z,y)$ defined in Lemma \ref{lem:Green} are both of at most polynomial growth when $y, z \to +\infty$;
\item The functions $\mathcal{K}$  and $\partial_y \mathcal{K}$ defined by \eqref{eq:extension} are both rapidly decreasing when $y, \lambda \to +\infty$.  
\end{itemize} 

To apply the previous theorem, we define $m(\xi)=\Phi(|\xi|^2)$ and to simplify the discussion we will consider only the Mikhlin's condition, since the conclusion of the discussion will be the same for H\"ormander's case. It is clear that if $\mathcal{K}(y,|\xi|^2)$ satisfies a Mikhlin-type condition (i.e. enough derivatives in $\xi$ decay away from zero), then the Leibniz rule (together with dominated convergence) implies that so is $m$. Therefore, Theorem \ref{T:HormanderMult} gives that $\Phi(-L)$ is bounded on $L^p$ for every $p \in (1,\infty)$. Notice however that the constant in the estimate depends on the dimension $n$. In connection to other probabilistic (martingale) proofs of Theorem \ref{T:HormanderMult}, we mention that in  \cite[Theorem 1.1]{McC},  McConnell gave a probabilistic proof under the assumption that $|\gamma|\leq n+1$ in the context of functions taking values in a Banach spaces with the UMD (unconditional
martingale difference sequence) property. There too, the constants depend on the dimension even in the case when the UMD space is $\R$. 

%
%
In this section, we will prove a Gundy-Varopoulos representation for  $\Phi(-L)$ therefore  allowing us to prove the boundedness in $L^p$ with an {\sl explicit} constant independent of the dimension (see Theorem \ref{thm:W}). In some special cases of diffusion $\mathcal B$, one can compute explicitly $G$ and $\mathcal K$, allowing to check the previous decay assumptions and then leading to some special multipliers.

\subsection{Construction of the martingale transform associated to a multiplier}

We consider then the multiplier defined for $ f\in C_0^\infty(M)$ by 
\[
Wf=\Phi (-L) f,
\]
where
\begin{equation}\label{eq:W}
\Phi(\lambda) = \int_0^{+\infty} G(+\infty,y) \partial_y\mathcal{K} (y , \lambda)^2  a(y)^2 dy. 
\end{equation}

We again note that due to the growth assumptions on $G$ and $\partial_y \mathcal{K}$, the function $\Phi$ is bounded. Moreover, using the extension function $U_f$ defined in \eqref{extension_def}, one can see that $W$ satisfies the following integration by parts formula: for every $f,g \in C_0^\infty(M)$,
\[
\int_M g Wf d\mu=\int_0^{+\infty} \int_M  \partial_y U_f (y,x)  \partial_y U_g (y,x) G(+\infty,y)a(y)^2 \, d\mu(x) dy.
\]
Let now $\mathbb P^{x,y}$ be the probability measure associated with the stochastic process $(X_t,\eta_t)$ starting at the point $(x,y)$ with $x\in M$ and $y>0$, define a measure $\mathbb P^y$  by 
\[
\mathbb P^y((X_{t \wedge \tau},\eta_{t \wedge \tau})\in \Theta)=\int_{M} \mathbb P^{x,y}((X_{t \wedge \tau}, \eta_{t \wedge \tau})\in \Theta)d\mu(x)
\]
for any Borel set $\Theta\in  M\times \R^+$. In particular, since $\mu$ is invariant for the semigroup for any Borel set $\Theta\in M$, $\mathbb P^y(X_{\tau}\in \Theta)=\mu(\Theta)$. From this it follows that  for any nonnegative (or integrable function) $f$ on $M$, we have $\mathbb{E}^y(f(X_{\tau}))=\int_M f(x)d\mu(x)$. In particular,  for any $1<p<\infty$, the $L^p$ norm of the random variable $Z=f(X_{\tau})$ equals the $L^p(M,\mu)$ norm of the function $f$.  In this case we will just write 
$\|f(X_{\tau})\|_{L^p}=\|f\|_{L^p(M,\mu)}$.

\begin{thm}\label{thm:TiSijW-rep 1}
We have the following Gundy-Varopoulos type representation for $W$: for every $ f\in C_0^\infty (M)$ and $x \in M$,
\begin{align*}
W f (x) &=\frac12 \lim_{y_0 \to +\infty} \mathbb{E}^{y_0} \left( e^{\int_0^{\tau} V(X_u)du} \int_0^{\tau} e^{-\int_0^{s} V(X_u)du} \partial_{y} U_f (X_s,\eta_s) a(\eta_s) \, d\beta_s \mid X_\tau =x  \right).
\end{align*}
\end{thm}

\begin{proof}

Note first that as a consequence of Lemma \ref{lem:eqB}, since $X$ and $\eta$ are independent, we have
\begin{equation}\label{eq:killedP}
\mathbb{E}^{y} \left(  \int_0^\tau F (X_s, \eta_s) ds \right)=\int_0^{+\infty} \int_{M}G(y,z) F (x,z) d\mu(x) dz,
\end{equation}
for all Borel function $F$ on $M\times \mathbb R^+$ such that $\int_M\int_0^{\infty} \frac{h(z)}{a(z)^2}|F(x,z)|dzd\mu(x)<\infty$.

Let $f,g \in C_0^\infty(M)$. We observe
 that
\[
M_{\tau}^g=e^{\int_0^{\tau} V(X_u)du}g(X_\tau).
\] 
Applying It\^o's formula \eqref{eq:Ito} for  $M_{\tau}^g$ and 
 the It\^o isometry, one has
\begin{align*}
\int_{M}g(x)\mathbb{E}^{y_0}& \left( e^{\int_0^{\tau} V(X_u)du} \int_0^{\tau} e^{-\int_0^{s} V(X_u)du} \partial_y U_f (X_s,\eta_s) a(\eta_s) d\beta_s \mid X_{\tau}=x  \right) d\mu(x) 
\\ &=
\mathbb{E}^{y_0} \left(g(X_\tau) e^{\int_0^{\tau} V(X_u)du} \int_0^{\tau} e^{-\int_0^s V(X_u)du} \partial_y U_f (X_s,\eta_s) a(\eta_s) d\beta_s \right)   
\\&=
2 \mathbb{E}^{y_0} \left( \int_0^{\tau}\partial_{y} U_g (X_s,\eta_s) \partial_y U_f (X_s,\eta_s) a(\eta_s)^2 ds   \right) \\
&=
2\int_0^{+\infty}\int_{M} G(y_0,y)  \partial_{y} U_g (x,y) \partial_y U_f (x,y)a(y)^2 \, d \mu(x) \, dy,
\end{align*}
where the last inequality is due to \eqref{eq:killedP}.
Since $U_f (x ,y)=\mathcal{K} (y , -L)f (x)$ and $L$ is self-adjoint, we have
\[
\int_{M}  \partial_{y} U_g (x,y) \partial_y U_f (x,y) d\mu (x)
=\int_{M} g(x)  \partial_y \mathcal{K} (y , -L)\partial_y \mathcal{K} (y , -L)  f (x)d\mu (x)
\]
and therefore
\begin{multline*}
\int_0^{+\infty}\int_{M} G(y_0,y)  \partial_{y} U_g (x,y) \partial_y U_f (x,y) d\mu(x) a(y)^2 \, dy
\\
=\int_{M} g(x) \int_0^{+\infty}G(y_0,y)  \partial_y \mathcal{K} (y , -L)\partial_y \mathcal{K} (y , -L)  f (x)  a(y)^2 dy\, d\mu(x).
\end{multline*}

We conclude that for every $g \in C_0^\infty(M)$
\begin{align*}
 & \int_{M}g(x)\mathbb{E}^{y_0} \left( e^{\int_0^{\tau} V(X_u)du} \int_0^{\tau} e^{-\int_0^{s} V(X_u)du} \partial_y U_f (X_s,\eta_s) a(\eta_s) d\beta_s \mid X_{\tau}=x  \right) d\mu(x) \\
 =&2  \int_{M} g(x) \int_0^{+\infty}G(y_0,y)  \partial_y \mathcal{K} (y , -L)\partial_y \mathcal{K} (y , -L)  f (x) a(y)^2 dy\, d\mu(x). 
\end{align*}

Therefore,
\begin{align*}
& \mathbb{E}^{y_0} \left( e^{\int_0^{\tau} V(X_u)du} \int_0^{\tau} e^{-\int_0^{s} V(X_u)du} \partial_y U_f (X_s,\eta_s) a(\eta_s) d\beta_s \mid X_{\tau}=x  \right) \\
=&2 \int_0^{+\infty}G(y_0,y)  \partial_y \mathcal{K} (y , -L)\partial_y \mathcal{K} (y , -L)  f (x) a(y)^2 dy.
\end{align*}
The conclusion follows by taking the limit $y_0 \to +\infty$.
\end{proof}

\subsection{Boundedness in $L^p$}

 \begin{thm}\label{thm:W}
The operator $W$ defined by \eqref{eq:W} is bounded in $L^p$. Moreover,  if the potential $V\equiv0$,   we have for every $f \in L^p(M,\mu)$
\[
\|W f\|_p\le \frac12 (p^*-1)\|f\|_p.
\]
And, if the potential $V$ is not zero, then
\[
\|W f\|_p\le \frac32 (p^*-1)\|f\|_p.
\]
\end{thm}

\begin{proof}
Let $f \in C_0^\infty(M)$. One can write
\begin{align*}
W f (x) &=\frac12 \lim_{y_0 \to +\infty} \mathbb{E}^{y_0} \left( e^{\int_0^{\tau} V(X_u)du} \int_0^{\tau} e^{-\int_0^{s} V(X_u)du}  dY_s \mid X_\tau =x  \right),
\end{align*}
where
\[
Y_t =\int_0^t \partial_{y} U_f (X_s,\eta_s) a(\eta_s) d\beta_s.
\]
If $V\equiv0$, the martingale $Y$ is differentially subordinate to the martingale $U_f (X,\eta)$, see Lemma \ref{martin}. One can conclude by using Theorem \ref{thm:BO15} with $V_t=0$.

Next we deal with the case $V\neq 0$ and adapt a method used in the proof of Theorem 1.1 in \cite{BBC}. If $V \neq 0$, then $U_f (X,\eta)$ is not a martingale anymore. However,  the martingale $Y$ is  differentially subordinate to the martingale
\[
N_t:=\sum_{i=1}^n\int_0^{t \wedge \tau}  \bigg( \sum_{j=1}^nv_{ij}\partial_{x_j} \bigg) U_f (X_s, \eta_s) dB_s^i+ \int_0^{t \wedge \tau} \partial_{y} U_f (X_s,\eta_s) a(\eta_s) d\beta_s,
\]
where we recall that $(v_{ij}(x))$ is the square root of the matrix $(\sigma_{ij}(x))$.
From It\^o's formula, one also has
\[
N_t=U_f (X_{t \wedge \tau},\eta_{t \wedge \tau})-U_f (X_{0},\eta_{0})-2\int_0^{t \wedge \tau} (\Delta+\mathcal B)U_f (X_s, \eta_s) ds.
\]

We now note that from Theorem \ref{stinga-torrea}
\[
(\Delta+\mathcal B)U_f =-V U_f.
\]
Therefore,
\[
N_t=U_f (X_{t \wedge \tau},\eta_{t \wedge \tau})-U_f (X_{0},\eta_{0}) +2 \int_0^{t \wedge \tau} V(X_s) U_f (X_s, \eta_s) ds.
\]
Suppose  $f\ge 0$. Then, the above equality yields that $U_f (X_{t \wedge \tau},\eta_{t \wedge \tau})$ is a non-negative sub-martingale. 
 It follows from Lenglart-L\'epingle-Pratelli \cite[Theorem 3.2, part 3)]{MR580107}
 that 
\begin{align*}
\norm{U_f (X_{0},\eta_{0}) - 2\int_0^{\tau} V(X_s) U_f (X_s, \eta_s) ds}_p &\le p\,\N{ U_f (X_\tau,\eta_\tau) }_p
=p\,\N{f(X_\tau)}_p= p\,\N{f}_p, 
\end{align*}
where we used the fact that the $L^p$ norm of the random variable $f(X_{\tau})$ equals the $L^p(M,\mu)$ norm of the function $f$, as pointed earlier. 
For a general $f$, since $V$ is non-positive, we note that
\[
 \left| U_f (X_{0},\eta_{0}) - 2\int_0^{\tau} V(X_s) U_f (X_s, \eta_s) ds \right|
 \le U_{|f|}  (X_{0},\eta_{0}) - 2\int_0^{ \tau} V(X_s) U_{|f|} (X_s, \eta_s) ds.
\]
This yields that we always have
\[
\| N_\tau \|_p \le (p+1) \N{f}_p.
\]
Therefore by Lemma \ref{thm:BO15}
\[
\| Y_\tau \|_p\le (p^*-1) \N{N_\tau}_p \le (p+1) (p^*-1) \N{f}_p.
\]
We conclude from the Gundy-Varopoulos type representation of $W$ and contraction of the conditional expectation that 
\[
\|W f\|_p\le \frac12 (p+1)(p^*-1)\|f\|_p.
\]
For $1<p \le 2$, this gives the inequality in the statement of the theorem. The similar inequality in the range $p >2$ is obtained by using the fact that the $L^p$ adjoint operator of $W$ is itself.
\end{proof}

\subsection{Specific choices for the vertical diffusion}

In this section we give explicit expressions of the operator $W$ depending upon  the choices of the function $b$. It suffices to compute the function $\mathcal K(y,\lambda)$ defined in \eqref{eq:extension} and the Green function associated to the operator $\mathcal B$.  For computations, it may be easier to use an alternative representation of the multiplier. Recall that
\begin{equation}\label{eq:W2}
\Phi(\lambda) = \int_0^{+\infty} G(+\infty,y) \partial_y\mathcal{K} (y , \lambda)^2   a(y)^2 dy. 
\end{equation}
\begin{lem}
\[
\Phi(\lambda) = \frac{1}{2} -\lambda \int_0^{+\infty} G(+\infty,y) \mathcal{K} (y , \lambda)^2   dy. 
\]
\end{lem}

\begin{proof}
From the definition of $\mathcal K(y,\lambda)$ in \eqref{eq:extension}, we have that  $\mathcal{B}\mathcal{K} (\cdot,\lambda )=\lambda \mathcal{K} (\cdot,\lambda )$, see page 18 in \cite{BorSal}. Moreover,  $\mathcal{K} (y,\lambda )$ converges to 1 as $y \to 0$ and  $\mathcal{K} (y,\lambda ) \to 0$ as $y \to \infty$. Thus, using integration by parts one obtains
\begin{align*}
&\int_0^{+\infty} G(+\infty,y) \partial_y\mathcal{K} (y , \lambda)^2   a(y)^2 dy
\\&= \mathcal{K} (y , \lambda)\partial_y\mathcal{K} (y , \lambda)  G(+\infty,y)  a(y)^2\Big\vert_0^{+\infty}
- \int_0^{+\infty} \mathcal{K} (y , \lambda) \partial_y(a(y)^2G(+\infty,y) \partial_y\mathcal{K} (y , \lambda))  dy
\\ &= -\int_0^{+\infty} \mathcal{K} (y , \lambda)\partial_y\mathcal{K} (y , \lambda)  dy- \int_0^{+\infty} \mathcal{K} (y , \lambda) G(+\infty,y) \mathcal B\mathcal{K} (y , \lambda)  dy
\\ &= \frac{1}{2} -\lambda \int_0^{+\infty} G(+\infty,y) \mathcal{K} (y , \lambda)^2   dy. 
\end{align*}
Here the second equality follows from the definition of $G(+\infty, y)$ which gives 
\[
\partial_y(a(y)^2G(+\infty,y))=\partial_y\left(h(y)\int_0^y\frac{dw}{h(w)}\right)=a(y)^2\frac{h'(y)}{h(y)}G(+\infty, y)+1.
\]
\end{proof}

\subsubsection{Brownian motion with negative drift}

Assume that $a(y)=\sigma$ and $b(y)=-2m$, where $m \ge 0$ and $\sigma >0$. One computes that 
\[
\mathcal{K} (y , \lambda)=e^{-\frac{y}{\sigma}  \left(\sqrt{\lambda +\frac{m^2}{\sigma^2}}-\frac{m}{\sigma} \right)},
\]
see for instance \cite[page 69]{BGL} or \cite[page 295]{BorSal}.
Taking $h(x)=e^{-\frac{2m}{\sigma^2} (x-1)}$ in Lemma \ref{lem:Green} yields 
\begin{align*}
G(y,z) =\frac{1}{2m} e^{-\frac{2m}{\sigma^2} z} \left(e^{\frac{2m}{\sigma^2} (y \wedge z)} -1\right).
\end{align*}
In particular,
\[
G(+\infty,z)=\frac{1}{2m}(1-e^{-\frac{2m}{\sigma^2} z} ).
\]
For this choice of $a$ and $b$, we can now rewrite the operator  $W$ defined in \eqref{eq:W} as
\[
Wf=\frac1{4} \left( I-\frac{m}{\sigma} \left(-L +\frac{m^2}{\sigma^2}\right)^{-1/2} \right)f.
\]

We are now in position to conclude the proof of Theorem \ref{gen multiplier}. We actually state a slightly stronger version including the potential $V$.
\begin{theorem}\label{bounded 2}
Let $\Phi:[0, +\infty) \to \mathbb{C}$ be a bounded Borel function.
If there exists a finite complex Borel measure $\alpha$ on $\mathbb{R}_{\ge 0}$ such that for every $x \in [0,+\infty)$,
\[
\Phi(x)= \int_{0} ^{+\infty} \left( 1-\frac{m}{\sqrt{m^2+x}}\right) d \alpha (m),
\]
then, for every  $p>1$ and $f \in L^p (M,\mu)$,
\[
\left\| \Phi(-L) f\right\|_p \le 6(p^*-1) | \alpha |(\mathbb{R}_{\ge 0})  \|f\|_p.
\]
If $V\equiv0$, this bound can be improved to
\[
\left\| \Phi(-L) f\right\|_p \le 2(p^*-1) | \alpha |(\mathbb{R}_{\ge 0})  \|f\|_p.
\]
\end{theorem}

\begin{proof}
It follows from the expression of $W$ and Theorem \ref{thm:W} that for every $m \ge 0$
\[
\norm{  \left( I-m \left(-L +m^2\right)^{-1/2} \right)f }_p  \le 6(p^*-1)  \|f\|_p.
\]
Thus, we have
\[
\left\| \Phi(-L) f\right\|_p \le 6(p^*-1) | \alpha |(\mathbb{R}_{\ge 0})  \|f\|_p.
\]
When $V\equiv0$ the bound can be improved thanks to Theorem \ref{thm:W}.
\end{proof}

\subsubsection{Bessel processes}

Assume that $a(y)=1$ and $b(y)=\frac{\gamma}{y}$, where $-1<\gamma<1$. Set $\gamma = 1-2s$, then one computes that
\[
\mathcal{K}_s (y , \lambda)=\frac{2^{1-s}}{\Gamma(s)} y^s \lambda^{s/2}K_s ( y \lambda^{1/2}),
\]
where $K_s(x)$ is the MacDonald function (Bessel function of the second kind) defined as follows
\begin{equation}\label{McD1}
K_s(x)=\frac{1}{2} \left(\frac{x}{2} \right)^s \int_0^{+\infty} \frac{e^{-t-\frac{x^2}{4t}}}{t^{1+s}} dt.
\end{equation}
Taking $h(x)=y^\gamma$ in Lemma \ref{lem:Green}, it follows that  $G(+\infty,y)=\frac y{2s}$. Using the formula
\begin{equation}\label{McD2}
 \int_0^{+\infty}  y^{\alpha-1}  K_\nu ( y )^2   dy= \frac{\sqrt{\pi}}{4 \Gamma \left( \frac{1+\alpha}{2}\right)} \Gamma \left( \frac{\alpha}{2} \right)\Gamma \left( \frac{\alpha}{2} -\nu \right)\Gamma \left( \frac{\alpha}{2}+\nu \right)
 \end{equation} 
 that holds for $\alpha > 2\nu >0$ (see \cite{AbraSteg}), one gets $W f = \frac{1}{2(2-\gamma)} f$.

\begin{remark}[Bessel processes with negative drift]
Interestingly, some  (partially) explicit computations may also be carried out in a class of processes extensively studied by Pitman \& Yor in \cite{MR620995}. Those processes introduced by S. Watanabe in \cite{MR365731}, and generalizing the Bessel processes, are sometimes called Bessel processes with negative (or descending) drift. Assume that $a(y)=1$ and
\[
b(y)=\frac{2\nu+1}{y}-2\delta \frac{K_{1+\nu}(\delta y)}{K_\nu(\delta y)}
\]
with $\nu,\delta >0$. In that case, direct computations show that the multiplier $W$ takes the form

\[
W=\frac{1}{2} +\frac{L(-L+\delta^2)^{\nu}}{\delta^{2\nu}}  \int_0^{+\infty}  K_\nu \left(y \sqrt{-L+\delta^2}\right)^2  \frac{I_\nu (\delta y)}{K_\nu(\delta y)} y \, dy.
\]

\end{remark}

\section{Generalized Riesz transforms}

In this section we will construct other operators arising from martingale transforms. We work with all the assumptions and notations of Section \ref{sec2.1} but assume furthermore that the operator $\Delta$ admits a representation 
\[
\Delta=-\sum_{i=1}^n \X_i^* \X_i,
\]
where the $\X_i$'s are smooth vector fields on $M$,  $\X_i^*$ denotes the formal adjoint of $\X_i$ with respect to $\mu$. We denote as before $L=\Delta +V$, where $V:M\to \R$ is as before the non-positive bounded smooth potential and  $(P_t)_{t\ge 0}$ the heat semigroup with generator $L$. Note that we can write   
\[
\Delta=\sum_{i=1}^n \X_i^2+\X_0,
\]
for some smooth vector field $\X_0$ and that from the celebrated H\"ormander theorem \cite{MR222474}, a sufficient condition for the local subellipticity of $\Delta$ is then that the vector fields $\X_0,\X_1, \cdots, \X_n$ satisfy the bracket generating condition, see page 6 in \cite{baudoin2018geometric}.

Let $(X_t)_{t\ge 0}$ be the diffusion process on $M$ with generator  $\sum_{i=1}^n \X_i^2+\X_0$ starting from the distribution $\mu$. We assume that $(X_t)_{t\ge 0}$ is nonexplosive  and can  be constructed  via the Stratonovitch stochastic differential equation
\[
dX_t=\X_0(X_t)dt+\sum_{i=1}^n \X_i(X_t)\circ dB_t^i, 
\]
where $B_t=(B_t^1,\cdots,B_t^n)$ is a Brownian motion on $\R^n$ with generator $\Delta=\sum_{i=1}^n \frac{\partial^2}{\partial x_i^2}$. 

\begin{remark}
For instance, if $M$ is a Euclidean space, then a standard sufficient condition so that $(X_t)_{t\ge 0}$ is nonexplosive  and can  be constructed  via the above  Stratonovitch stochastic differential equation is that the vector fields $\X_0,\X_1, \cdots, \X_n$ have globally bounded derivatives, see for instance \cite[Theorem 6.29]{Baudoin2014}. In that case, if the vector fields $\X_0,\X_1, \cdots, \X_n$ also satisfy the bracket generating condition, then $\Delta$ is locally subelliptic and essentially self-adjoint.
\end{remark}

As before, see \eqref{SDEeta1}, we will consider the one-dimensional diffusion  on $(0,+\infty)$ given by 
\[
d\eta_t=b(\eta_t) dt+  a(\eta_t) d\beta_t, \quad t < \tau
\]
where $\beta_t$ is a Brownian motion on $\R$ with $\mathbb E(\beta_t^2)=2t$ which is independent from $(X_t)_{t \ge 0}$.

\subsection{Operators arising from martingale transforms}

We introduce now the class of operators under consideration. For any $1\le i,j\le n$, we consider the operators 
\begin{equation}\label{eq:Ti}
T_i f= \int_0^{+\infty} a(y) G(+\infty,y) \partial_y\mathcal{K} (y , -L)  \X_i \mathcal{K} (y , -L) f dy,
\end{equation}
and
\begin{equation}\label{eq:Sij}
S_{ij} f=\int_0^{+\infty} G(+\infty,y) \mathcal{K} (y , -L)\X_j^*\X_i \mathcal{K} (y , -L)f dy.
\end{equation}

\begin{thm}\label{thm:TiSijW-rep 2}
We have the following Gundy-Varopoulos type representations: for every $ f\in C_0^\infty (M)$ and $x \in M$
\begin{align*}
T_i f (x) &=\frac12 \lim_{y_0 \to +\infty} \mathbb{E}^{y_0} \left( e^{\int_0^{\tau} V(X_u)du} \int_0^{\tau} e^{-\int_0^{s} V(X_u)du} \X_i U_f (X_s,\eta_s)  d\beta_s \mid X_\tau =x  \right),
\\
S_{ij} f (x) &=\frac12 \lim_{y_0 \to +\infty} \mathbb{E}^{y_0} \left( e^{\int_0^{\tau} V(X_u)du} \int_0^{\tau} e^{-\int_0^{s} V(X_u)du} \X_i U_f (X_s,\eta_s)  dB^j_s \mid X_\tau =x  \right).
\end{align*}
\end{thm}

\begin{proof}
The proof is similar to the proof of Theorem \ref{thm:TiSijW-rep 1}, we present it for completeness. It suffices to show the first expression of $T_i$. The same proof also works for  $S_{ij}$. Let $f,g \in C_0^\infty(M)$. We recall 
$
M_{\tau}^g=e^{\int_0^{\tau} V(X_u)du}g(X_\tau).
$
Applying It\^o's formula for $M_t^g$ gives that
\begin{equation}\label{eq:ito}
\begin{split}
M_{t}^g=U_g(X_0, \eta_0) +\sum_{i=1}^n \int_0^{t \wedge \tau}e^{\int_0^{s}V(X_u)du} \X_i U_g(X_s, \eta_s) dB_s^i \\+  \int_0^{t \wedge \tau} e^{\int_0^{s}V(X_u)du} \partial_yU_g(X_s, \eta_s)a(\eta_s)d\beta_s.
\end{split}
\end{equation}
By the It\^o isometry, one has
\begin{align*}
\int_{M}g(x)\mathbb{E}^{y_0}& \left( e^{\int_0^{\tau} V(X_u)du} \int_0^{\tau} e^{-\int_0^{s} V(X_u)du} \X_i U_f (X_s,\eta_s)   d\beta_s \mid X_{\tau}=x  \right) d\mu(x) 
\\ &=
\mathbb{E}^{y_0} \left(g(X_\tau) e^{\int_0^{\tau} V(X_u)du} \int_0^{\tau} e^{-\int_0^s V(X_u)du} \X_i U_f (X_s,\eta_s)  d\beta_s \right) 
\\&=
2 \mathbb{E}^{y_0} \left( \int_0^{\tau}\partial_{y} U_g (X_s,\eta_s) \X_i U_f (X_s,\eta_s) a(\eta_s) ds   \right) \\
&=
2\int_0^{+\infty}\int_{M} a(y) G(y_0,y)  \partial_{y} U_g (x,y) \X_i U_f (x,y) d \mu(x) \, dy,
\end{align*}
where the last equality is due to \eqref{eq:killedP}.
Since $U_f (x ,y)=\mathcal{K} (y , -L)f (x)$ and $L$ is self-adjoint, then
\[
\int_{M}  \partial_{y} U_g (x,y) \X_i U_f (x,y) d\mu (x)
=\int_{M} g(x)  \partial_y \mathcal{K} (y , -L)\X_i \mathcal{K} (y , -L)  f (x)d\mu (x)
\]
and 
\begin{multline*}
\int_0^{+\infty}\int_{M} a(y)G(y_0,y)  \partial_{y} U_g (x,y) \X_i U_f (x,y) d\mu(x) \, dy
\\
=\int_{M} g(x) \int_0^{+\infty} a(y) G(y_0,y)  \partial_y \mathcal{K} (y , -L)\X_i \mathcal{K} (y , -L)  f (x) dy\, d\mu(x).
\end{multline*}
The rest of the proof thus immediately follows.

\end{proof}

\subsection{Boundedness in $L^p(\mu)$}

\begin{cor}\label{thm:Ti}
Let  $1\le i\le n$. For all $1<p<\infty$ and all $f\in C_0^{\infty}(M)$ we have
\[
\|T_i f\|_p\le \frac32(p^*-1) \|f\|_p.
\]
Moreover, if the potential $V\equiv0$, then
\[
\|T_i f\|_p\le \frac12(p^*-1)  \|f\|_p.
\]
\end{cor}
\begin{proof}
When $V\equiv0$, the operator $T_i$ can be rewritten as 
\[
T_if(x)= \frac12\lim_{y_0 \to +\infty} \mathbb{E}^{y_0} \left(\int_0^{\tau}   A_i(\X, \partial_y)^T U_f (X_s,\eta_s)\cdot (dB_s,d\beta_s) \mid X_\tau =x  \right),
\]
where $A_i$ is an $(n+1)\times (n+1)$ matrix with the entry $a_{(n+1),i}=1$ and otherwise $0$. 
Notice that the martingale   
\(
\int_0^{t \wedge \tau} A_i(\X, \partial_y)^T U_f (X_s,\eta_s)\cdot (dB_s,d\beta_s) 
\)
is differentially subordinate to the martingale $U_f (X_{t\wedge \tau},\eta_{t\wedge \tau})$.
It follows from Theorem \ref{thm:BW95} that
\[
\|T_i f\|_p\le \frac12(p^*-1) \|f\|_p.
\]

When $V\ne 0$, then the same method as for the proof of Theorem \ref{thm:W} implies the desired estimate.
 \end{proof}

\begin{cor}\label{thm:Sij}
Let  $1\le i,j\le n$. For all $1<p<\infty$ and all $f\in C_0^{\infty}(M)$ we have
\[
\|S_{ij} f\|_p\le \frac32(p^*-1) \|f\|_p.
\]
Moreover, if the potential $V\equiv0$, then
\[
\|S_{ij} f\|_p\le \frac12 (p^*-1)\|f\|_p.
\]
\end{cor}
\begin{proof}
Similarly as for $T_i$, when $V\equiv0$ one can write $S_{ij}$  as 
\[
S_{ij}f(x)= \frac12\lim_{y_0 \to +\infty} \mathbb{E}^{y_0} \left(\int_0^{\tau}  A_{ij} (\X, \partial_y)^T U_f (X_s,\eta_s)\cdot (dB_s,d\beta_s) \mid X_\tau =x  \right),
\]
where $A_{ij}$ is an $(n+1)\times (n+1)$ matrix with the entries $a_{i,j}=1$ and otherwise $0$. Since  
\(
\int_0^{t \wedge \tau} A_i(\X, \partial_y)^T U_f (X_s,\eta_s)\cdot (dB_s,d\beta_s) 
\)
is differentially subordinate to $U_f (X_{t\wedge \tau},\eta_{t\wedge \tau})$, then Theorem \ref{thm:BW95} yields
\[
\|S_{ij} f\|_p\le \frac12(p^*-1) \|f\|_p.
\]
 \end{proof}

\subsection{Euclidean spaces and Lie groups of compact type}

We now apply our results to the case of Euclidean spaces and Lie groups of compact type.  In those cases,  for the transforms we are interested in, the operators $\X_i$'s and $\X_i^*$'s do commute with $L$. As a consequence, one has

\begin{equation}\label{eq:Ti2}
T_i f= \int_0^{+\infty} a(y) G(+\infty,y) \partial_y\mathcal{K} (y , -L)   \mathcal{K} (y , -L)\, \X_i fdy ,
\end{equation}
and
\begin{equation}\label{eq:Sij2}
S_{ij} f=\int_0^{+\infty}  G(+\infty,y) \mathcal{K} (y , -L)^2  \, \X_j^*\X_ifdy.
\end{equation}

\subsubsection{Brownian motion with negative drift as vertical diffusion}

Consider the Euclidean spaces $\R^n$. We assume that the potential $V$ in the operator $L$ is null.
In this case, $\X_i=\partial_{x_i}$ commutes with the Laplace operator $L=\Delta$. 

\begin{lemma}\label{lem:TiSijRd}
Let $1\le i,j\le n$ and $\sigma>0$, $m\ge 0$. For the choice $\mathcal B = \sigma^2 \frac{\partial^2 }{\partial y^2} - 2m \frac{\partial }{\partial y}$, one has
\[
T_i f=-\frac1{4} \left(-\Delta +\frac{m^2}{\sigma^2}\right)^{-1/2} \partial_{x_i}f
\]
and
\[
S_{ij} f =-\frac1{4} \left( \sqrt{-\Delta +\frac{m^2}{\sigma^2}} -\frac{m}{\sigma}  \right)^{-1} \left(-\Delta +\frac{m^2}{\sigma^2} \right)^{-1/2}   \, \partial_{x_i} \partial_{x_j}f.
\]
\end{lemma}

\begin{proof}
Since $\partial_{x_i}$ commutes with $\Delta$, the operator $T_i$ becomes
\begin{align*}
T_i f&=  \left(\int_0^{+\infty} \frac{1-e^{\frac{2m}{\sigma^2}y}}{2m} \left(\sqrt{-\Delta +\frac{m^2}{\sigma^2}}-\frac{m}{\sigma}\right) e^{-\frac{2y}{\sigma}\sqrt{-\Delta+\frac{m^2}{\sigma^2}}} dy \right) \partial_{x_i} f
\\ &=
\frac1{2m}\left(\sqrt{-\Delta +\frac{m^2}{\sigma^2}}-\frac{m}{\sigma}\right)  \left( \frac\sigma{2} \left(\sqrt{-\Delta+\frac{m^2}{\sigma^2}}\right)^{-1} -\left( \frac{2}{\sigma} \sqrt{-\Delta+\frac{m^2}{\sigma^2}} -\frac{2m}{\sigma^2}\right)^{-1} \right)
\\ &=
\frac{\sigma}{4m}\left(\left(\sqrt{-\Delta +\frac{m^2}{\sigma^2}}-\frac{m}{\sigma}\right) \left(\sqrt{-\Delta +\frac{m^2}{\sigma^2}}\right)^{-1}- I \right)  \, \partial_{x_i} f
\\ &=
-\frac1{4} \left(-\Delta +\frac{m^2}{\sigma^2}\right)^{-1/2} \partial_{x_i}f.
\end{align*}
Similarly,  we have
\begin{align*}
S_{ij} f &= \left(  \int_0^{+\infty} \frac{1-e^{\frac{2m}{\sigma^2}y}}{2m} e^{-\frac{2y}{\sigma} \sqrt{-\Delta +\frac{m^2}{\sigma^2}}}  dy \right)  \, \partial_{x_i} \partial_{x_j}f
\\ &= 
-\frac\sigma{4m}\left( \left( \sqrt{-\Delta +\frac{m^2}{\sigma^2}} -\frac{m}{\sigma} \right)^{-1}-  \left(\sqrt{-\Delta +\frac{m^2}{\sigma^2}}\right)^{-1}\right)  \, \partial_{x_i} \partial_{x_j}f
\\ &= 
-\frac1{4} \left( \sqrt{-\Delta +\frac{m^2}{\sigma^2}} -\frac{m}{\sigma}  \right)^{-1} \left(-\Delta +\frac{m^2}{\sigma^2} \right)^{-1/2}   \, \partial_{x_i} \partial_{x_j}f.
\end{align*}
\end{proof}

We obtain therefore:

\begin{prop}\label{prop:RTconstant}
Let $1\le i,j\le n$ and $m\ge 0$, $\sigma>0$. Then
\begin{equation}\label{firstRiesz} 
\left\| \left(-\Delta +\frac{m^2}{\sigma^2}\right)^{-1/2}  \partial_{x_i}f\right\|_p \le  \cot\left(\frac{\pi}{2p^*}\right) \|f\|_p,
\end{equation} 
\begin{align}\label{secondRiesz}
\left\| \left( \sqrt{-\Delta +\frac{m^2}{\sigma^2}} -\frac{m}{\sigma}  \right)^{-1} \left(-\Delta +\frac{m^2}{\sigma^2} \right)^{-1/2}  \, \partial_{x_i} \partial_{x_j} f \right\|_p\le (p^*-1)\|f\|_p.
\end{align}
\end{prop}

\begin{proof}

Recall the Gundy-Varopoulos type representation of $T_i$ in Theorem \ref{thm:TiSijW-rep 2}:
\[
T_i f (x) =\frac{1}{2} \lim_{y_0 \to +\infty} \mathbb{E}^{y_0} \left(\int_0^{\tau}  \partial_{x_i} U_f (X_s,\eta_s)  d\beta_s \mid X_\tau =x  \right).
\] 
In a similar way we also  have
\[
T_if(x)=-\frac12 \lim_{y_0 \to +\infty} \mathbb{E}^{y_0} \left(  \int_0^{\tau} \sigma\, \partial_{y} U_f (X_s,\eta_s) dB_s^i \mid X_\tau =x \right).
\]
To see this, observe that  by It\^o's formula \eqref{eq:ito} and the It\^o isometry, one has for any $y_0>0$
\begin{align*}
\int_{M}g(x)\mathbb{E}^{y_0}& \left( \int_0^{\tau} \sigma\,  \partial_{y} U_f (X_s,\eta_s) dB_s^i \mid X_\tau =x  \right) d\mu(x) 
=
\mathbb{E}^{y_0} \left(g(X_\tau) \int_0^{\tau} \sigma\,   \partial_{y} U_f (X_s,\eta_s) dB_s^i\right) 
\\&
=2 \mathbb{E}^{y_0} \left( \int_0^{\tau} \sigma\, \partial_{x_i} U_g (X_s,\eta_s) \partial_{y} U_f (X_s,\eta_s) ds   \right) \\
&=
2\int_0^{+\infty}\int_{M} \sigma\, G(y_0,y)  \partial_{x_i} U_g (x,y) \partial_{y} U_f (x,y) d \mu(x) \, dy
\\&=
-2\int_{M} g(x)\int_0^{+\infty} \sigma\, G(y_0,y)  \partial_y\mathcal{K} (y , -L)  \partial_{x_i} \mathcal{K} (y , -L) f(x) dy d \mu(x).
\end{align*}
Taking the limit $y_0\to \infty$ gives the second expression of $T_if$.
Therefore
\begin{equation}\label{eq:fullvaro} 
T_if(x)= \frac14\lim_{y_0 \to +\infty} \mathbb{E}^{y_0} \left(\int_0^{\tau} A_i(\X, \partial_y)^T U_f (X_s,\eta_s)\cdot (dB_s,d\beta_s) \mid X_\tau =x  \right),
\end{equation} 
where $A_i$ is an $(n+1)\times (n+1)$ matrix with the entries $a_{(n+1),i}=1$, $a_{i,(n+1)}=-\sigma$ and otherwise $0$.

With this matrix we claim that (1) the martingale 
\(
N_t:=\int_0^{t \wedge \tau} A_i(\X, \partial_y)^T U_f (X_s,\eta_s)\cdot (dB_s,d\beta_s) 
\)
is differentially  subordinate  to $M_t^f=U_f (X_{t\wedge \tau},\eta_{t\wedge \tau})$ and that (2) the two martingales are orthogonal.  That is, $\langle N,M^f\rangle_t=0$. 
To verify (1) recall that  that 
\[
M_{t}^f=f(X_0,\eta_0)+\sum_{j=1}^n \int_0^{t\wedge \tau}  \partial_{x_j} U_f(X_s, \eta_s)dB_s^j+ \int_0^{t\wedge \tau} \partial_yU_f(X_s, \eta_s)\sigma d\beta_s,
\]
and 
\[
\langle M^f\rangle_t=2\sum_{j=1}^n \int_0^{t\wedge \tau}  \left(\partial_{x_j} U_f(X_s, \eta_s)\right)^2ds+
2\int_0^{t\wedge \tau} \left(\partial_yU_f(X_s, \eta_s)\right)^2\sigma^2 ds
\]
Similarly, 

\[
 N_t=\int_0^{t\wedge \tau} \partial_{x_i} U_f(X_s, \eta_s)d\beta_s-\int_0^{t\wedge \tau} \partial_yU_f(X_s, \eta_s)\sigma dB_{s}^{i} 
\]
and
$$ 
\langle N\rangle_t=2\int_0^{t\wedge \tau}  \left(\partial_{x_i} U_f(X_s, \eta_s)\right)^2ds+
2\int_0^{t\wedge \tau} \left(\partial_yU_f(X_s, \eta_s)\right)^2\sigma^2 ds.
$$
Thus 
$$
\langle M^f\rangle_t-\langle N\rangle_t=2\sum_{j\neq i}\int_0^{t\wedge \tau} \left(\partial_{x_j} U_f(X_s, \eta_s)\right)^2ds, 
$$
which is  a nondecreasing and nonnegative function of $t$. This proves the differential  subordination property. 

To prove the orthogonality we note that 
\begin{align*}
\langle N,M^f\rangle_t =&-\sigma\int_0^{t\wedge \tau}\partial_{x_i} U_f(X_s, \eta_s)\partial_yU_f(X_s, \eta_s) d\langle B^i\rangle_s\\
&+
\sigma\int_0^{t\wedge \tau}\partial_y U_f(X_s, \eta_s)\partial_{x_i}U_f(X_s, \eta_s) d\langle \beta \rangle_s\\
=& \,0. 
\end{align*}

We can now apply \eqref{eq:fullvaro} and the martingale inequality in Theorem \ref{thm:BW95}. It follows that

\[
\|T_i f\|_p\le \frac{1}{4}  \cot\left(\frac{\pi}{2p^*}\right) \|f\|_p
\]
and hence by Lemma \ref{lem:TiSijRd},
\[
\left\| \left(-\Delta +\frac{m^2}{\sigma^2}\right)^{-1/2}  \partial_{x_i}f\right\|_p =4\|T_i f\|_p\le \cot\left(\frac{\pi}{2p^*}\right) \|f\|_p.
\]

On the other hand
\begin{align*}
S_{ij} f(x)
&=\frac12 \lim_{y_0 \to +\infty} \mathbb{E}^{y_0} \left(\int_0^{\tau} \partial_{x_j} U_f (X_s,\eta_s) dB^i_s \mid X_\tau =x  \right)
\\&=
\frac14 \lim_{y_0 \to +\infty} \mathbb{E}^{y_0} \left(\int_0^{\tau}  A_{ij}(\X, \partial_y)^T U_f (X_s,\eta_s) \cdot (dB_s,d\beta_s) \mid X_\tau =x  \right),
\end{align*}
where $A_{ij}$ is an $(n+1)\times (n+1)$ matrix with the entries $a_{i,j}=a_{j,i}=1$ and otherwise $0$. Observe that the matrix norm of $A_{ij}$ is $1$, then the martingale   
\(
\int_0^{t \wedge \tau} A_{ij}(\X, \partial_y)^T U_f (X_s,\eta_s)\cdot (dB_s,d\beta_s) 
\)
is differentially subordinate  to $U_f (X_{t\wedge \tau},\eta_{t\wedge \tau})$.

It follows from Theorem \ref{thm:BW95} that 
\[
\|S_{ij} f\|_p\le \frac14(p^*-1) \|f\|_p
\]
and again by Lemma \ref{lem:TiSijRd},
\[
\left\| \left( \sqrt{-\Delta +\frac{m^2}{\sigma^2}} -\frac{m}{\sigma}  \right)^{-1} \left(-\Delta +\frac{m^2}{\sigma^2} \right)^{-1/2}  \, \partial_{x_i} \partial_{x_j} f \right\|_p\le (p^*-1)\|f\|_p.
\]
\end{proof}

\begin{remark}\label{limsigma}  
The degenerate case $\sigma=0$ corresponds to the case where $d\eta_t=-2mdt$, i.e. $\eta_t =2m(T-t)$,  where $T=\tau$ is deterministic.  This gives the space-time process introduced  in \cite{BM}. 
 That is, taking the limit $\sigma \to 0$  in \eqref{secondRiesz}  with $m>0$ fixed, we recover the main results in \cite{BM} and \cite{NazVol} concerning the norms of second order Riesz transforms.   Namely  that 
\begin{equation}\label{bestsecondRiesz1}
\|2R_iR_j\|_p=\left\| 2(-\Delta)^{-1}    \, \partial_{x_i} \partial_{x_j} \right\|_p\le  (p^*-1)\|f\|_p.
\end{equation} 
Indeed, formally as $\sigma \to 0$ one has
\begin{equation*}
\left( \sqrt{-\Delta +\frac{m^2}{\sigma^2}} -\frac{m}{\sigma}  \right)^{-1} \left(-\Delta +\frac{m^2}{\sigma^2} \right)^{-1/2}  \partial_{x_i} \partial_{x_j}f \rightarrow 2(-\Delta)^{-1} \partial_{x_i} \partial_{x_j}f.
\end{equation*}

In the same way,  we obtain that 
\begin{equation}\label{bestsecondRiesz2}
\|R_i^2-R_j^2\|_p=\left\| (-\Delta)^{-1}    \, \partial_{x_i} \partial_{x_i}-(-\Delta)^{-1}    \, \partial_{x_j} \partial_{x_j} \right\|_p\le (p^*-1)\|f\|_p.
\end{equation}
For $i \neq j$, the bounds $(p^*-1)$ in \eqref{bestsecondRiesz1} and \eqref{bestsecondRiesz2} were shown to be the best possible in \cite{GeiSmiSak}. 

When $i=j$, combining our methods here with the martingale inequalities from \cite{BanOse1} we would obtain (again by letting $\sigma\to 0$) the inequality 
\begin{equation}\label{bestsecondRiesz3}
\|R_i^2 f\|_p=\left\|(-\Delta)^{-1}    \, \partial_{x_i} \partial_{x_i} \right\|_p\le c_p\|f\|_p 
\end{equation} 
first proved in \cite{BanOse1}. 
Here, $c_p$ is the best constant found in \cite{Choi} for non-symmetric martingale transforms, i.e. martingale transforms where the predictable sequences take values in $[0, 1]$. Although the constant is not as nice as Burkholder's  $(p^*-1)$ for general martingale transforms, it can be estimated quite well and in particular it satisfies for large $p$,
$$
c_p\approx \frac{p}{2}+ \frac{1}{2}\log\left(\frac{1+e^{-2}}{2}\right) +\frac{\alpha_2}{p},
$$
where
$$\alpha_2=\left[\log\left(\frac{1+e^{-2}}{2}\right)\right]^2+\frac{1}{2}\log\left(\frac{1+e^{-2}}{2}\right)-2\left(\frac{e^{-2}}{1+e^{-2}}\right)^{2}.$$ 
In addition it follows easily from Burkholder's inequality  (see \cite[Theorem 4.1]{Choi}) that  the constant $c_p$ satisfies the bounds 
\begin{equation}\label{choibound}
\max\left(1, \frac{p^*}{2}-1\right)\leq c_p\leq \frac{p^*}{2}.
\end{equation}

It is also important to mention here that the constant $c_p$ in \eqref{bestsecondRiesz3} is also best possible.  This was first proved in \cite{BanOse1}.  For this and more general results related to second order Riesz transforms, we refer the reader to \cite{BanOse1} and particularly Theorems 1.4, 1.5, and Corollary 1.3.  
%

On the complex plane $\C$, which we identify with $\R^2$, the Beurling-Ahlfors operator is defined by $Bf=(-\Delta)^{-1}{\partial}^2f$, where $\partial $ is the Cauchy-Riemann operator 
$
\partial f=\frac{\partial f}{\partial x_1}-i \frac{\partial f}{\partial x_2}. 
$ 
A longstanding open problem with connections to several areas of analysis, PDE's and geometry,  known as   Iwaniec's conjecture  \cite[p.129]{AIM},  asserts that  
\begin{equation}\label{BA-Conj}
\|Bf\|_p\leq (p^*-1)\|f\|_p, \,\,\, 1<p<\infty, 
\end{equation} 
for all $f:\C\to \C$, $f\in C_{0}^{\infty}(\C)$. 

That the constant $(p^*-1)$ in \eqref{BA-Conj} cannot be improved has been known for many years, see \cite{AIM}. Writing the operator $B=R_1^2-R_2^2+2iR_1R_2$ in terms of Riesz transforms we see from the above discussion that the real and imaginary parts of the Beurling-Ahlfors operator have the same norm as the martingale transforms, that is, the  Burkholder constant $p^*-1$. This is in fact the same for  functions taking values in a Banach space with the UMD property, see \cite[Theorem 1.1]{GeiSmiSak}.    Applying  \eqref{bestsecondRiesz1} and \eqref{bestsecondRiesz2}  leads to the estimate $\|Bf\|_p\leq 2(p^*-1)\|f\|_p$. 
This bound was first proved in \cite{NazVol} and \cite{BM} and later improved to $1.575(p^*-1)$ in \cite{BanJan} by martingale inequalties applied to martingales satisfying certain orthogonality properties.  For a detailed discussion of these results we refer the reader to \cite{Ban}. The key point in \cite{BM} and \cite{BanJan} is to use the martingale techniques applied to the space-time process.  That is, build the martingales on the process $(X_t, T-t)$ which arise from the heat extension rather than the Poisson extension.   Given that we now know that the process $(X_t, T-t)$ arises from the general Poison extensions treated in this paper by letting  $\sigma\to 0$, it is natural to wonder if further progress on Iwaniec's conjecture can be made by better choices of the vertical diffusion $\eta_t$.


It is also interesting to note that as $\sigma\to +\infty$ (or equivalently $m \to 0$), we get the inequality 
$$
\norm{ \left(-\Delta\right)^{-1/2}\partial_{x_i}f}_p \le \cot\left(\frac{\pi}{2p^*}\right) \|f\|_p,
$$
which is sharp as shown in \cite{IM} and \cite{BW}.  Thus, the inequalities \eqref{firstRiesz} and \eqref{secondRiesz} are both sharp in the sense that there is no universal constant $C<1$ independent of $\sigma$ and $m$ for which the first holds with $C\cot\left(\frac{\pi}{2p^*}\right)$ on the right hand side and the second  with $C(p^*-1)$.  
\end{remark}

%
%

The previous methods can be applied  to Lie groups of compact type. Let $G$ be a Lie group of compact type with Lie algebra $\mathfrak{g}$.  We recall that $G$ is called a Lie group of compact type if its Lie algebra $\mathfrak{g}$ admits an $\mathbf{Ad}$-invariant inner product. In that case, this equips $G$ with a bi-invariant metric. Note that Euclidean spaces are examples of Lie groups of compact type so that this framework is a generalization of the Euclidean one.

We consider  an orthonormal basis $\X_1,\cdots, \X_n$ of $\mathfrak{g}$.  In this setting the Laplace-Beltrami operator can be written as 
\[
L= \sum_{i=1}^n \X_i^2.
\]
Observe that $L$ is essentially self-adjoint on the space  of smooth and compactly supported functions. Moreover, $\X_i^*=-\X_i$ and $\X_i$ commutes with $L$. In the same manner as Euclidean spaces case we obtain the following result. 
\begin{prop}\label{prop:RTconstantLie}
Let G be a Lie group of compact type endowed with a bi-invariant
Riemannian structure. Let $1\le i,j\le n$ and $m\ge 0$. Then
\[
\left\|(-L +m^2)^{-1/2} \X_i f\right\|_p \le \cot\left(\frac{\pi}{2p^*}\right) \|f\|_p,
\]
\[
\left\| \left( \sqrt{-L +m^2} -m \right)^{-1} \left(-L +m^2\right)^{-1/2}   \, \frac{1}{2} ( \X_i \X_j  +\X_j \X_i )f \right\|_p\le (p^*-1)\|f\|_p.
\]
\end{prop}

The proof of Theorem \ref{compact lie 1} easily  follows from
 Proposition \ref{prop:RTconstantLie} by integrating with respect to $m$.

\subsubsection{Bessel process as vertical diffusion}

In this section, we work on Lie groups of compact type endowed with bi-invariant
Riemannian structures.

\begin{lemma}
Let G be a Lie group of compact type endowed with a bi-invariant
Riemannian structure. Let $1\le i,j\le n$.  For the choice $\mathcal B = \frac{\partial^2 }{\partial y^2} +b(y) \frac{\partial }{\partial y}$ with  $b(y)=\frac{\gamma}{y}$, $-1<\gamma<1$, one has
\[
T_i f=-\frac{\pi^2 \Gamma(4s)}{2^{8s} s^2 \Gamma(s)^4} (-L)^{-1/2} \X_i
\]
and
\[
S_{ij} f =\frac{s}{2s+1}  (-L)^{-1} \X_i \X_j,
\]
where $\gamma = 1-2s$.
\end{lemma}

\begin{proof}
Using the Bessel process, we recall that 
\[
\mathcal{K}_s (y , \lambda)=\frac{2^{1-s}}{\Gamma(s)} y^s \lambda^{s/2}K_s ( y \lambda^{1/2}),
\]
with  $\gamma = 1-2s$ and that $G(+\infty,y)=\frac{y}{2s}$. Therefore
\begin{align*}
\int_0^{+\infty} G(+\infty,y) \partial_y\mathcal{K} (y , \lambda)   \mathcal{K} (y , \lambda) dy &=\frac{1}{2} \int_0^{+\infty} G(+\infty,y) \partial_y ( \mathcal{K} (y , \lambda)^2)    dy \\
 &=-\frac{1}{4s} \frac{2^{2-2s}}{\Gamma(s)^2} \int_0^{+\infty}  y^{2s} \lambda^{s}K_s ( y \lambda^{1/2})^2 dy \\
 &= -\frac{\pi^2 \Gamma(4s)}{2^{8s} s^2 \Gamma(s)^4} \, \lambda^{-1/2}.
\end{align*}
Similarly
\begin{align*}
\int_0^{+\infty} G(+\infty,y)  \mathcal{K} (y , \lambda)^2 dy&=\frac{1}{2s} \frac{2^{2-2s}}{\Gamma(s)^2} \int_0^{+\infty}  y^{2s+1} \lambda^{s}K_s ( y \lambda^{1/2})^2 dy=\frac{s}{2s+1} \lambda^{-1}.
\end{align*}
\end{proof}

Using the Bessel process as a vertical diffusion, one deduces therefore:

\begin{prop}\label{prop:RTconstantLie2}
Let G be a Lie group of compact type endowed with a bi-invariant
Riemannian structure. Let $1\le i,j\le n$. Then for every $s \in (0,1)$
\[
\left\|(-L )^{-1/2} \X_i f\right\|_p \le \frac{2^{8s} s^2 \Gamma(s)^4}{4\pi^2 \Gamma(4s)}  \cot\left(\frac{\pi}{2p^*}\right) \|f\|_p,
\]
\[
\left\| \frac{1}{2} ( \X_i \X_j  +\X_j \X_i )(-L)^{-1}f \right\|_p\le \frac{2s+1}{4s} (p^*-1)\|f\|_p.
\]
\end{prop}

\begin{proof} Since as before the martingale representations for $T_i f$ and $S_{ij} f$ give that 
$\|T_i f\|_p\leq \frac{1}{4}\cot\left(\frac{\pi}{2p^*}\right) \|f\|_p$ and $\|S_{ij} f\|_p\leq \frac{1}{4}(p^*-1)\|f\|_p$,  both inequalities follow immediately. 

\end{proof} 

Of course the constant $\frac{2s+1}{4s}$ is best for $s\to1$ which corresponds to 0-dimensional Bessel process as a vertical diffusion. On the other hand the constant $\frac{2^{8s} s^2 \Gamma(s)^4}{4\pi^2 \Gamma(4s)}$ is optimal for $s=1/2$ which corresponds to 1-dimensional Bessel process (=Brownian motion) as a vertical diffusion.

\subsection{Generalized Riesz transform on vector bundles}

We consider the framework introduced in  Section 3.1 of \cite{BBC}.
Let $M$ be a $n$-dimensional  smooth complete Riemannian
manifold and let $\mathcal{E}$ be a finite-dimensional vector
bundle over $M$. We denote by $\Gamma ( M,
\mathcal{E} )$ the space of smooth  sections of this bundle. Let  $\nabla$
denote a metric connection on $\mathcal{E}$. We consider an operator on $\Gamma ( M,
\mathcal{E} )$  that can be written as
\[
\mathcal{L}=\mathcal{F}+\nabla_{0}+ \sum_{i=1}^n \nabla_{i}^2,
\]
where
\[
\nabla_i=\nabla_{\X_i}, \quad 0 \le i \le n,
\]
and the $\X_i$'s  are smooth vector fields on $M$ and 
$\mathcal{F}$  is a smooth symmetric and non positive potential (that is a smooth section of the
bundle $\mathbf{End}(\mathcal{E})$). We assume that $\mathcal{L}$ is locally subelliptic, non-positive and essentially self-adjoint on the space $\Gamma_0(M,\mathcal{E})$ of smooth and compactly supported sections. We consider then  a first order differential operator $d_a$ on $\Gamma ( M,
\mathcal{E} )$ that can be written as
\[
d_a=\sum_{i=1}^n a_i \nabla_{\X_i},
\]
where $a_1,\cdots,a_n$ are smooth sections of the bundle $\mathbf{End}(\mathcal{E})$. Assume that $d_a$ commutes with $\mathcal L$, i.e.
\[
d_a \mathcal{L} \eta =\mathcal{L} d_a \eta,\quad \eta \in \Gamma ( M,
\mathcal{E} ),
\]
and that 
\[
\| d_a \eta \|^2 \le C \sum_{i=1}^n \| \nabla_{\X_i} \eta \|^2, \quad \eta \in \Gamma ( M,
\mathcal{E} ),
\]
for some constant $C\ge 0$. 

The following theorem can then be proved by combining the techniques of this paper with the analysis performed in Section 3.1 of \cite{BBC}.

\begin{theorem} \label{JKNM}
Let $\Phi: (0,+\infty) \to \mathbb{C}$ be a complex Borel function. If there exists a finite complex Borel measure $\alpha$ on $\mathbb{R}_{\ge 0}$ such that for every $x \in (0,+\infty)$,
\[
\Phi(x)= \int_{0} ^{+\infty} \frac{d \alpha (m)} {\sqrt{ x +m}},
\]
then, for every  $p>1$ and   $\eta \in \Gamma_0 ( M,\mathcal{E}  )$
\[
\|  \Phi(-\mathcal L ) \, d_a \eta \|_p \le 6 C (p^*-1) |\alpha |(\mathbb{R}_{\ge 0})  \| \eta \|_p.
\]
\end{theorem}

Theorem \ref{Riesz Hodge} follows then from the previous theorem as in Section 3.2 of \cite{BBC}.

\subsection*{Acknowledgements}  We are grateful to the anonymous referee for the very careful reading and the many comments, including corrections which greatly improved the paper.  

\bibliographystyle{amsplain}
\bibliography{multiplier-1.bib}

\end{document}